\documentclass[12pt]{amsart}  

\usepackage{amssymb}
\usepackage{amscd}
\usepackage{amsmath}

\newtheorem{thm}{Theorem}[section]
\newtheorem*{thm-intro}{Theorem}

\newtheorem{prop}[thm]{Proposition}
\newtheorem{lemma}[thm]{Lemma}

\theoremstyle{remark}
\newtheorem{remark}[thm]{Remark}

\theoremstyle{definition}

\numberwithin{equation}{section}

\DeclareMathOperator{\Hom}{Hom}
\DeclareMathOperator{\Aut}{Aut}
\newcommand{\Cat}{\operatorname{Cat}}
\newcommand{\id}{{\mathtt{Id}}}
\newcommand{\ev}{\operatorname{ev}}
\DeclareMathOperator{\MC}{MC}
\DeclareMathOperator{\ad}{\mathtt{ad}}

\begin{document}

\title{Deligne groupoid revisited}

\author[P.Bressler]{Paul Bressler}
\address {Departamento de Matem\'aticas, Universidad de Los Andes} \email{paul.bressler@gmail.com}

\author[A.Gorokhovsky]{Alexander Gorokhovsky}
\address{Department of Mathematics, UCB 395,
University of Colorado, Boulder, CO~80309-0395, USA}
\email{Alexander.Gorokhovsky@colorado.edu}

\author[R.Nest]{Ryszard Nest}
\address{Department of Mathematics,
Copenhagen University, Universitetsparken 5, 2100 Copenhagen, Denmark}
 \email{rnest@math.ku.dk}

\author[B.Tsygan]{Boris Tsygan}
\address{Department of
Mathematics, Northwestern University, Evanston, IL 60208-2730, USA}
\email{b-tsygan@northwestern.edu}

\begin{abstract}
We show that for a differential graded Lie algebra $\mathfrak{g}$ whose components vanish in degrees below $-1$ the nerve of the Deligne $2$-groupoid is homotopy  equivalent to the simplicial set of $\mathfrak{g}$-valued differential forms introduced by V.~Hinich \cite{H1}.
\end{abstract}

\thanks{
A. Gorokhovsky was partially supported by NSF grant DMS-0900968. B. Tsygan was partially
supported by NSF grant DMS-0906391. R. Nest was partially supported by the Danish National Research Foundation through the Centre
for Symmetry and Deformation (DNRF92)}
\maketitle
\section{Introduction}
The principal result of the present note compares two spaces (simplicial sets) naturally associated with a nilpotent differential graded Lie algebra (DGLA) subject to certain restrictions. Our interest in this problem has its origins in formal deformation theory of associative algebras and, more generally, algebroid stacks (\cite{BGNT}). The results of the present note are used in \cite{BGNT2} to deduce a quasi-classical description of the deformation theory of a gerbe from the formality theorem of M.~Kontsevich.

To a nilpotent DGLA $\mathfrak{h}$ which satisfies the additional condition
\begin{equation}\label{vanishing condition}
\text{$\mathfrak{h}^i = 0$ for $i<-1$}
\end{equation}
P.~Deligne \cite{Del} and, independently, E.~Getzler \cite{G1} associated a (strict) $2$-groupoid which we denote $\MC^2(\mathfrak{h})$ and refer to as the Deligne $2$-groupoid.

Our principal result (Theorem \ref{thm: nerve is sigma}) compares the simplicial nerve $\mathfrak{N}\MC^2(\mathfrak{h})$ of the $2$-groupoid $\MC^2(\mathfrak{h})$, $\mathfrak{h}$ a nilpotent DGLA satisfying \eqref{vanishing condition}, to another simplicial set, denoted $\Sigma(\mathfrak{h})$, introduced by V. Hinich \cite{H1}:
\begin{thm-intro}{(Main theorem)}
Suppose that $\mathfrak{h}$ is a nilpotent DGLA such that $\mathfrak{h}^i = 0$ for $i<-1$. Then, the simplicial sets $\mathfrak{N}\MC^2(\mathfrak{h})$ and $\Sigma(\mathfrak{h})$ are homotopy equivalent.
\end{thm-intro}

In the case when the nilpotent DGLA $\mathfrak{h}$ satisfies $\mathfrak{h}^i = 0$ for $i<0$ and, consequently, $\MC^2(\mathfrak{h})$ is an ordinary groupoid a homotopy equivalence between $\Sigma(\mathfrak{h})$ and the nerve of $\MC^2(\mathfrak{h})$ was constructed by V.~Hinich in \cite{H1}.

Differential grade Lie algebras satisfying \eqref{vanishing condition} arise in formal deformation theory of algebraic structures such as Lie algebras, commutative algebras, associative algebras to name a few. In what follows we shall concentrate on the latter example. Let $k$ denote an algebraically closed field of characteristic zero. For an associative algebra $A$ over $k$ the shifted Hochschild cochain complex $C^\bullet(A)[1]$ has a canonical structure of a DGLA under the Gerstenhaber bracket; we denote this DGLA by $\mathfrak{g}(A)$ for short. Suppose that $\mathfrak{m}$ is a nilpotent commutative $k$-algebra (without unit). Then, $\mathfrak{g}(A)\otimes_k\mathfrak{m}$ is a nilpotent DGLA which satisfies \eqref{vanishing condition}. Thus, the Deligne $2$-groupoid $\MC^2(\mathfrak{g}(A)\otimes_k\mathfrak{m})$ is defined. For an Artin $k$-algebra $R$ with maximal ideal $\mathfrak{m}_R$ the $2$-groupoid $\MC^2(\mathfrak{g}(A)\otimes_k\mathfrak{m}_R)$ is naturally equivalent to the $2$-groupoid of $R$-deformations of the algebra $A$. In this sense the DGLA $\mathfrak{g}(A)$ controls the formal deformation theory of $A$.

The reason for considering the space $\Sigma(\mathfrak{h})$ is that it is defined not just for a DGLA (V. Hinich, \cite{H1}), but, more generally, for any nilpotent $L_\infty$ algebra (E.~Getzler, \cite{G1}). Homotopy invariance properties of the functor $\Sigma$ (Proposition \ref{prop:homotopy invariance}), the theory of J.W.~Duskin (\cite{D}) and the theorem above yield the following result. If $\mathfrak{h}$ is a DGLA satisfying \eqref{vanishing condition}, $\mathfrak{g}$ is a $L_\infty$ algebra $L_\infty$-quasi-isomorphic to $\mathfrak{h}$ and $\mathfrak{m}$ is a nilpotent commutative $k$-algebra, then $\mathfrak{N}\MC^2(\mathfrak{h}\otimes_k\mathfrak{m})$ is homotopy equivalent to $\Sigma(\mathfrak{g}\otimes_k\mathfrak{m})$. Thus, the $2$-groupoid $\MC^2(\mathfrak{h}\otimes_k\mathfrak{m})$ can be reconstructed, up to equivalence, from the space $\Sigma(\mathfrak{g}\otimes_k\mathfrak{m})$. The situation envisaged above arises naturally. Any DGLA $\mathfrak{h}$ is $L_\infty$-quasi-isomorphic to an $L_\infty$ algebra with trivial univalent operation (the differential).

The paper is organized as follows. In Section \ref{s:hom type of a strict 2-gr} we review various  constructions  of nerves of $2$-groupoids and their properties. In section \ref{section: Homotopy types} we recall the definitions of the functor $\Sigma$ (\ref{ss: sigma}) and of the Deligne $2$-groupoid (\ref{ss: Deligne}) and prove basic properties thereof. The proof of the main theorem (Theorem \ref{thm: nerve is sigma}) given in Section \ref{section: mc vs sigma} proceeds by exhibiting canonical homotopy equivalences from $\Sigma(\mathfrak{h})$ and $\mathfrak{N}\MC^2(\mathfrak{h})$ to a third naturally defined simplicial set.

\section{The homotopy type of a strict 2-groupoid}\label{s:hom type of a strict 2-gr}

\subsection{Nerves of simplicial groupoids}\label{sss:BvS}
\subsubsection{Simplicial groupoids}\label{sss:simplicial grpds}
In what follows a \emph{simplicial category} is a category enriched over the category of
simplicial sets. A small simplicial category consists of a set of objects and a simplicial set of
morphisms for each pair of objects.

A simplicial category $\mathtt{G}$ is a particular case of a simplicial object $[p] \mapsto
\mathtt{G}_p$ in $\Cat$ whose simplicial set of objects $[p] \mapsto N_0\mathtt{G}_p$
is constant.

A simplicial category is a simplicial groupoid if it is a groupoid in each (simplicial) degree.

\subsubsection{The na\"{\i}ve nerve}
Suppose that $\mathtt{G}$ is a simplicial category. Applying the nerve functor degree-wise
we obtain the bi-simplicial set $N\mathtt{G} \colon ([p],[q]) \mapsto N_q\mathtt{G}_p$
whose diagonal we denote by $\mathcal{N}\mathtt{G}$ and refer to as the
\emph{na\"{\i}ve nerve} of $\mathtt{G}$.


\subsubsection{The simplicial nerve}
For a simplicial category $\mathtt{G}$  the
\emph{simplicial nerve}, also known as the homotopy coherent nerve,
$\mathfrak{N}\mathtt{G}$ is represented by the cosimplicial object in $[p] \mapsto
\Delta_\mathfrak{N}^p \in Cat_\Delta$, i.e
\[
\mathfrak{N}_p\mathtt{G} = \Hom_{\Cat_\Delta}(\Delta_\mathfrak{N}^p, \mathtt{G}) .
\]
Here, $\Delta_\mathfrak{N}^p$ is the canonical free simplicial resolution of $[p]$ which
admits the following explicit description (\cite{C}).

The set of objects of $\Delta_\mathfrak{N}^p$ is $\{0,1,\ldots, p\}$. For $0 \leq i \leq j
\leq p$ the simplicial set of morphisms is given by $\Hom_{\Delta_{\mathfrak N}^p}(i,j) =
N\mathcal{P}(i,j)$. The category $\mathcal{P}(i,j)$ is a sub-poset of $2^{\{0,\ldots,p\}}$ (with the induced partial ordering whereby viewed as a category) given by
\[
\mathcal{P}(i,j) = \{ I\subset\mathbb{Z} \mid (i,j\in I)\, \&\, (k\in I \implies i\leqslant k\leqslant j)\} .
\]
The composition in $\Delta_{\mathfrak N}^p$ is induced by functors
\[
\mathcal{P}(i,j)\times \mathcal{P}(j,k)\to {\mathcal P}(i,k) \colon (I,J)
\mapsto I\cup J .
\]
In particular, $\Delta_{\mathfrak N}^0 = [0]$ and $\Delta_{\mathfrak
N}^1 = [1]$

We refer the reader to \cite{H2} for applications to deformation theory and to \cite{L} for the connection with higher category theory. The simplicial nerve of a simplicial groupoid is a Kan complex which reduces to the usual nerve for ordinary groupoids.

Since $\Delta_\mathfrak{N}^0 = [0]$ (respectively, $\Delta_\mathfrak{N}^1 = [1]$) it follows that $\mathfrak{N}_0\mathtt{G}$ (respectively, $\mathfrak{N}_1\mathtt{G}$) is the set of objects (respectively, the set of morphisms) of $\mathtt{G}_0$.

\subsubsection{Comparison of nerves}
We refer the reader to \cite{H2} for the definition of the canonical map of simplicial sets $\mathcal{N}\mathtt{G} \to \mathfrak{N}\mathtt{G}$. In what follows we will make use of the following result of loc. cit.

\begin{thm}[\cite{H2}]\label{thm: comparison of nerves}
For any simplicial groupoid $\mathtt{G}$ the canonical map $\mathcal{N}\mathtt{G} \to
\mathfrak{N}\mathtt{G}$ is an equivalence.
\end{thm}

\subsection{Strict 2-groupoids}

\subsubsection{From strict $2$-groupoids to simplicial groupoids}\label{sss:FStSG}
Suppose that $\mathtt{G}$ is a strict $2$-groupoid, i.e. a groupoid enriched over the
category of groupoids. Thus, for every $g, g^\prime \in \mathtt{G}$, we have the groupoid
$\Hom_\mathtt{G}(g, g^\prime)$ and the composition is strictly associative.

The nerve functor $[p] \mapsto N_p(\cdot) := \Hom_\mathrm{Cat}([p],\cdot)$
commutes with products. Let $\mathtt{G}_p$ denote the category with the same objects
as $\mathtt{G}$ and with morphisms defined by $\Hom_{\mathtt{G}_p}(g, g^\prime) =
N_p\Hom_\mathtt{G}(g, g^\prime)$; the composition of morphisms is induced by the
composition in $\mathtt{G}$. Note that the groupoid $\mathtt{G}_0$ is obtained from
$\mathtt{G}$ by forgetting the $2$-morphisms.

The assignment $[p] \mapsto \mathtt{G}_p$ defines a simplicial object in groupoids with
the constant simplicial set of objects, i.e. a simplicial groupoid which we denote by $\mathtt{\widetilde{G}}$.

\begin{lemma}\label{lemma:the two gothic Ns}
The simplicial nerve $\mathfrak{N}\mathtt{\widetilde{G}}$ admits the following explicit
description:
\begin{enumerate}
\item There is a canonical bijection between $\mathfrak{N}_0\mathtt{\widetilde{G}}$ and the set of
    objects of $\mathtt{G}$.

\item For $n\geq 1$ there is a canonical bijection between $\mathfrak{N}_n\mathtt{\widetilde{G}}$
    and the set of data of the form $((\mu_i)_{0\leq i\leq n}, (g_{ij})_{0\leq i < j\leq n},
    (c_{ijk})_{0\leq i < j<k\leq n})$, where $(\mu_i)$ is an $(n+1)$-tuple of objects of
    $\mathtt{G}$, $(g_{ij})$ is a collection of $1$-morphisms
    $g_{ij}\colon\mu_j\to \mu_i$ and $(c_{ijk})$ is a collection of $2$-morphisms
    $c_{ijk}:g_{ij}g_{jk}\to g_{ik}$ which satisfies
    \begin{equation} \label{eq:cece}
    c_{ijl}c_{jkl}=c_{ikl}c_{ijk}
    \end{equation}
    (in
    the set of $2$-morphisms $g_{ij}g_{jk}g_{kl}\to g_{il}$).
\end{enumerate}
For a morphism $f\colon [m] \to [n]$ in $\Delta$ the induced structure map
$f^*\colon\mathfrak{N}_n\mathtt{\widetilde{G}} \to \mathfrak{N}_m\mathtt{\widetilde{G}}$ is given (under
the above bijection) by $f^*((\mu_i),(g_{ij}),(c_{ijk})) = ((\nu_i),(h_{ij}),(d_{ijk}))$,
where $\nu_i = \mu_{f(i)}$, $h_{ij} = g_{f(i),f(j)}$, $d_{ijk} = c_{f(i),f(j),f(k)}$ (cf.
\cite{D}).
\end{lemma}
\begin{proof}
An $n$-simplex of $\mathfrak{N}\mathtt{\widetilde{G}}$ is the following collection of
data:
\begin{enumerate}
\item objects $\mu_0, \ldots, \mu_n$ of $\mathtt{G}$;
\item morphisms of simplicial sets $N\mathcal{P}(i,j)) \to N\Hom_\mathtt{G}(\mu_i,
    \mu_j))$ intertwining the maps induced on the nerves by composition functors
    $\mathcal{P}(i,j)\times\mathcal{P}(j,k) \to \mathcal{P}(i,k)$ and
    $\Hom_\mathtt{G}(\mu_i,  \mu_j)\times \Hom_\mathtt{G}(\mu_j,  \mu_k)\to
    \Hom_\mathtt{G}(\mu_i,  \mu_k)$.
\end{enumerate}
Since the nerve functor is fully faithful, the above data are equivalent to the following:
\begin{enumerate}
\item objects $\mu_0, \ldots, \mu_n$ of $\mathtt{G}$;

\item for any $I\in N_0\mathcal{P}(i,j)$, a 1-morphism $g_I:\mu_j \to \mu_i$ in
    $\mathtt{G}$;

\item for any morphism $J\to I$ in $\mathcal{P}(i,j)$, a 2-morphism $c_{IJ} \colon g_J\to
    g_I$, such that
\begin{equation}\label{eq:comp for c}
c_{IJ}c_{JK}=c_{IK}
\end{equation}
These data have to be compatible with the composition pairings
$\mathcal{P}(i,j)\times\mathcal{P}(j,k)\to {\mathcal P}(i,k)$ and
$\Hom_\mathtt{G}(\mu_i,  \mu_j)\times\Hom_\mathtt{G}(\mu_j,  \mu_k)\to
\Hom_\mathtt{G}(\mu_i,  \mu_k)$.
\end{enumerate}
Let $g_{ij} \colon\mu_j\to \mu_i$ denote the morphism $g_{\{i,j\}}.$ By compatibility with
compositions, if $I=\{i,i_1,\ldots, i_k, j\}$ then $g_I=g_{ii_1}\ldots g_{i_kj}$. Let
$c_{ijk}$ denote the two-morphism $c_{ \{i,j,k\}, \{i,k\}} \colon g_{ik}\to
g_{ij}g_{jk}.$ Now, by virtue of \eqref{eq:comp for c} and of compatibility with
compositions, $c_{ijk}$ satisfy the two-cocycle identity \eqref{eq:comp for c}
and determine  $c_{IJ}$ for any $I$, $J$.
\end{proof}

In what follows, for  a strict 2-groupoid $\mathtt{G}$,  we will denote by $\mathcal{N}\mathtt{G}$ (respectively $\mathfrak{N}\mathtt{G}$)
  the  na\"{\i}ve  (respectively simplicial) nerve   of the associated simplicial groupoid $\mathtt{\widetilde{G}}$.

\section{Homotopy types associated with $L_\infty$-algebras}\label{section: Homotopy types}

\subsection{$L_\infty$-algebras}
We follow the notation of \cite{G1} and refer the reader to loc. cit. for details.

Recall that an $L_\infty$-algebra is a graded vector space $\mathfrak{g}$ equipped with
operations
\[
\textstyle{\bigwedge}^k \mathfrak{g} \to \mathfrak{g}[2-k] \colon x_1\wedge\ldots\wedge x_k \mapsto [x_1,\ldots,x_k]
\]
defined for $k = 1,2,\ldots$. which satisfy a sequence of Jacobi identities.

It follows from the Jacobi identities that the unary operation $[.]\colon \mathfrak{g} \to
\mathfrak{g}[1]$ is a differential, which we will denote by $\delta$.

An $L_\infty$-algebra is \emph{abelian} if all operations with valency two and higher (i.e. all
operations except for $\delta$) vanish. In other words, an abelian $L_\infty$-algebra is a
complex. An $L_\infty$-algebra structure with vanishing operations of valency three and
higher reduces to a structure of a DGLA.

The \emph{lower central series} of an $L_\infty$-algebra $\mathfrak{g}$ is the canonical
decreasing filtration $F^\bullet\mathfrak{g}$ with $F^i\mathfrak{g} = \mathfrak{g}$ for $i
\leq 1$ and defined recursively for $i\geq 1$ by
\[
F^{i+1}\mathfrak{g} = \sum_{k=2}^\infty \sum_{\substack{i = i_1+\dots + i_k \\ i_k \leqslant i}} [F^{i_1}\mathfrak{g},\ldots,F^{i_k}\mathfrak{g}] .
\]

An $L_\infty$-algebra is \emph{nilpotent} if there exists an $i$ such that $F^i\mathfrak{g}
= 0$.

\subsubsection{Maurer-Cartan elements}
Suppose that $\mathfrak{g}$ is a nilpotent $L_\infty$-algebra. For $\mu\in\mathfrak{g}^1$ let
\begin{equation}\label{curva}
\mathcal{F}(\mu) = \delta\mu + \sum_{k=2}^\infty \frac1{k!} [\mu^{\wedge k}] .
\end{equation}
The element $\mathcal{F}(\mu)$ of $\mathfrak{g}^2$ is called the \emph{curvature} of $\mu$. For any $\mu \in \mathfrak{g}^1$ the curvature $\mathcal{F}(\mu)$ satisfies
the Bianchi identity (\cite{G1}, Lemma 4.5)
\begin{equation}\label{bianchi}
\delta\mathcal{F}(\mu) + \sum_{k=1}^\infty\frac1{k!}[\mu^{\wedge k},\mathcal{F}(\mu)] = 0 .
\end{equation}

An element $\mu \in \mathfrak{g}^1$ is called a \emph{Maurer-Cartan element} (of $\mathfrak{g}$) if it satisfies the condition
$\mathcal{F}(\mu) = 0$.
The set of Maurer-Cartan elements of $\mathfrak{g}$ will be denoted
$\MC(\mathfrak{g})$:
\[
\MC(\mathfrak{g}):= \{ \mu \in \mathfrak{g}^1 \mid \mathcal{F}(\mu) = 0 \}.
\]
The set $\MC(\mathfrak{g})$ is pointed by the distinguished element $0 \in \mathfrak{g}^1$.

Suppose that $\mathfrak{a}$ is an abelian $L_\infty$-algebra. Then,
\[
\MC(\mathfrak{a}) = Z^1(\mathfrak{a}) := \ker(\delta \colon \mathfrak{a}^1 \to \mathfrak{a}^2) ,
\]
hence is equipped with a canonical structure of an abelian group.

\subsubsection{Central extensions}\label{subsubsection: central extensions MC}
Suppose that $\mathfrak{g}$ is a $L_\infty$-algebra and $\mathfrak{a}$ is a
subcomplex of $(\mathfrak{g}, \delta)$ such that
$[\mathfrak{a},\mathfrak{g},\ldots,\mathfrak{g}] = 0$ for all $k \geq 2$. In this case we will say that \emph{$\mathfrak{a}$ is central in $\mathfrak{g}$}.

If $\mathfrak{a}$ is central in $\mathfrak{g}$, then there is a unique structure of an
$L_\infty$-algebra on $\mathfrak{g}/\mathfrak{a}$ such that the projection
$\mathfrak{g} \to \mathfrak{g}/\mathfrak{a}$ is a map of $L_\infty$-algebras. If $\mathfrak{g}$ is nilpotent, then so is $\mathfrak{g}/\mathfrak{a}$.

In what follows we assume that $\mathfrak{g}$ is a nilpotent $L_\infty$-algebra and $\mathfrak{a}$ is central in $\mathfrak{g}$.

\begin{lemma}\label{lemma: props of MC central extension}
{~}
\begin{enumerate}
\item The addition operation on $\mathfrak{g}^1$ restricts to a free action of the abelian group
$\MC(\mathfrak{a})$ on the set $\MC(\mathfrak{g})$.

\item The map $\MC(\mathfrak{g}) \to \MC(\mathfrak{g}/\mathfrak{a})$ is constant on the orbits of the action.

\item The induced map $\MC(\mathfrak{g})/\MC(\mathfrak{a}) \to \MC(\mathfrak{g}/\mathfrak{a})$ is
injective.
\end{enumerate}
\end{lemma}
\begin{proof}
Suppose that $\alpha \in \mathfrak{a}^1$ and $\mu \in \mathfrak{g}^1$. Since
$\mathfrak{a}$ is central in $\mathfrak{g}$, $[(\alpha + \mu)^{\wedge k}] =
[\mu^{\wedge k}]$ for $k \geq 2$ and $\mathcal{F}(\alpha + \mu) = \delta\alpha +
\mathcal{F}(\mu)$ (in the notation of \eqref{curva}). Therefore,
$\MC(\mathfrak{a}) + \MC(\mathfrak{g}) = \MC(\mathfrak{g})$. In other words, the
addition operation in $\mathfrak{g}^1$ restricts to an action of the abelian group
$\MC(\mathfrak{a})$ on the set $\MC(\mathfrak{g})$ which is obviously free. Since the map
$\MC(\mathfrak{g}) \to \MC(\mathfrak{g}/\mathfrak{a})$ is the restriction of the map $\mathfrak{g} \to \mathfrak{g}/\mathfrak{a}$ constant on the orbits of the
action, i.e. factors through $\MC(\mathfrak{g})/\MC(\mathfrak{a})$, and the induced map
$\MC(\mathfrak{g})/\MC(\mathfrak{a}) \to \MC(\mathfrak{g}/\mathfrak{a})$ is
injective.
\end{proof}

\subsubsection{The obstruction map}\label{subsubsection: The obstruction map}
The image of the map $\MC(\mathfrak{g}) \to \MC(\mathfrak{g}/\mathfrak{a})$ may be described in terms of the obstruction map \eqref{eqn:o2} which we construct presently.

If $\mu \in \mathfrak{g}^1$ and $\mu + \mathfrak{a}^1 \in
\MC(\mathfrak{g}/\mathfrak{a})$, then $\mathcal{F}(\mu + \mathfrak{a}^1) =
\mathcal{F}(\mu) + \delta\mathfrak{a}^1 \subset \mathfrak{a}^2$ and the Bianchi
identity \eqref{bianchi} reduces to $\delta\mathcal{F}(\mu + \mathfrak{a}^1) = 0$, i.e.
the assignment $\mu + \mathfrak{a}^1 \mapsto \mathcal{F}(\mu + \mathfrak{a}^1)$
gives rise to a well-defined map
\begin{equation}\label{eqn:o2}
o_2 \colon \MC(\mathfrak{g}/\mathfrak{a}) \to H^2(\mathfrak{a})
\end{equation}
(notation borrowed from \cite{GM}, 2.6).

\begin{lemma}
The sequence of pointed sets
\begin{equation}\label{ses MC central extension}
0 \to \MC(\mathfrak{g})/\MC(\mathfrak{a}) \to \MC(\mathfrak{g}/\mathfrak{a}) \xrightarrow{o_2} H^2(\mathfrak{a})
\end{equation}
is exact.
\end{lemma}
\begin{proof}
If $\mathcal{F}(\mu + \mathfrak{a}^1) \subset \delta\mathfrak{a}^1$, then there exists
$\alpha\in\mathfrak{a}^1$ such that $\mathcal{F}(\mu + \alpha) = 0$, i.e. $\mu +
\mathfrak{a}^1$ is in the image of $\MC(\mathfrak{g}) \to
\MC(\mathfrak{g}/\mathfrak{a})$.
\end{proof}

\subsection{The functor $\Sigma$}\label{ss: sigma}
In what follows we denote by $\Omega_n$, $n = 0, 1, 2, \ldots$ the commutative differential graded algebra over $\mathbb{Q}$ with generators $t_0,\ldots,t_n$ of degree zero and $dt_0, \ldots, dt_n$ of degree one subject to the relations $t_0 + \cdots + t_n = 1$ and $dt_0 + \cdots + dt_n = 0$. The differential $d \colon \Omega_n \to \Omega_n[1]$ is defined by $t_i \mapsto dt_i$ and $dt_i \mapsto 0$. The assignment $[n] \mapsto \Omega_n$ extends in a natural way to a simplicial commutative differential graded algebra.

\subsubsection{The simplicial set $\Sigma(\mathfrak{g})$}
For a nilpotent $L_\infty$-algebra $\mathfrak{g}$ and a non-negative integer $n$ let
\[
\Sigma_n(\mathfrak{g}) =\MC(\mathfrak{g}\otimes\Omega_n) .
\]
Equipped with structure maps induced by those of $\Omega_\bullet$ the assignment $n \mapsto \Sigma_n(\mathfrak{g})$ defines a simplicial set denoted $\Sigma(\mathfrak{g})$.

The simplicial set $\Sigma(\mathfrak{g})$ was introduced by Hinich in \cite{H1} for DGLA
and used by Getzler in \cite{G1} (where it is denoted $\MC_\bullet(\mathfrak{g})$) for
general nilpotent $L_\infty$-algebras.

\subsubsection{Abelian algebras}
If $\mathfrak{a}$ is an abelian algebra, then $\Sigma(\mathfrak{a})$ is given by
$\Sigma_n(\mathfrak{a}) = Z^1(\Omega_n\otimes\mathfrak{a}) =
Z^0(\Omega_n\otimes\mathfrak{a}[1])$ and has a canonical structure of a simplicial
abelian group. In particular, it is a Kan simplicial set.

Recall that the Dold-Kan correspondence associates to a complex of abelian groups $A$ a
simplicial abelian group $K(A)$ defined by $K(A)_n = Z^0(C^\bullet([n];A))$, the group of
cocycles of (total) degree zero in the complex of simplicial cochains on the $n$-simplex with
coefficients in $A$.

The integration map $\int\colon \Omega_n\otimes\mathfrak{a} \to
C^\bullet([n];\mathfrak{a})$ induces a homotopy equivalence
\begin{equation}\label{abelian integration}
\int\colon \Sigma(\mathfrak{a}) \to K(\mathfrak{a}[1]) .
\end{equation}
Thus, $\pi_i\Sigma(\mathfrak{a}) \cong H^{1-i}(\mathfrak{a})$.

\subsubsection{Central extensions}\label{subsubsection: central extensions sigma}
Suppose that $\mathfrak{g}$ is a nilpotent $L_\infty$-algebra and $\mathfrak{a}$ is a central subalgebra in $\mathfrak{g}$. Then, for $n = 0, 1, \ldots$,
$\Omega_n\otimes\mathfrak{a}$ is central in $\Omega_n\otimes\mathfrak{g}$.

\begin{lemma}{~}
\begin{enumerate}
\item The addition operation on $(\Omega_n\otimes\mathfrak{g})^1$ induces a principal action of the simplicial abelian group $\Sigma(\mathfrak{a})$ on the simplicial set $\Sigma(\mathfrak{g})$.

\item The map $\Sigma(\mathfrak{g}) \to \Sigma(\mathfrak{g}/\mathfrak{a})$ factors through $\Sigma(\mathfrak{g})/\Sigma(\mathfrak{a})$.

\item The induced map $\Sigma(\mathfrak{g})/\Sigma(\mathfrak{a}) \to \Sigma(\mathfrak{g}/\mathfrak{a})$ is injective.
\end{enumerate}
\end{lemma}
\begin{proof}
Follows from Lemma \ref{lemma: props of MC central extension} and the naturality properties of the constructions in \ref{subsubsection: central extensions MC}.
\end{proof}

For $n=0,1,\ldots$ the map $([n]\to[0])^\ast \colon \mathbb{Q} \to \Omega_n$ is a quasi-isomorphism, with the quasi-inverse provided by the map induced by any morphism $[0] \to [n]$. Therefore, the map $\mathfrak{a} \to \Omega_n\otimes\mathfrak{a}$ is a quasi-isomorphism as well. The induced isomorphisms $H^2(\mathfrak{a}) \cong H^2(\Omega_n\otimes\mathfrak{a})$ give rise to the isomorphism of the constant simplicial set $H^2(\mathfrak{a})$ and $n \mapsto H^2(\Omega_n\otimes\mathfrak{a})$.

The maps
\[
o_{2,n} \colon \Sigma_n(\mathfrak{g}/\mathfrak{a}) = \MC(\Omega_n\otimes\mathfrak{g}/\mathfrak{a}) \to H^2(\Omega_n\otimes\mathfrak{a}) \cong H^2(\mathfrak{a})
\]
assemble into the map of simplicial sets
\begin{equation}\label{eq: Obst und Gemuse 1}
o_2 \colon \Sigma(\mathfrak{g}/\mathfrak{a})
\to H^2(\mathfrak{a}) .
\end{equation}
which factors as $\Sigma(\mathfrak{g}/\mathfrak{a}) \to \pi_0\Sigma(\mathfrak{g}/\mathfrak{a}) \to H^2(\mathfrak{a})$.

Let
$\Sigma(\mathfrak{g}/\mathfrak{a})_0 = o_2^{-1}(0)$. Thus, by \eqref{ses MC central extension},
$\Sigma(\mathfrak{g}/\mathfrak{a})_0$ is a union of connected components of
$\Sigma(\mathfrak{g}/\mathfrak{a})$ equal to the range of the map
$\Sigma(\mathfrak{g})/\Sigma(\mathfrak{a}) \to \Sigma(\mathfrak{g}/\mathfrak{a})$.

It follows that the map $\Sigma(\mathfrak{g}) \to
\Sigma(\mathfrak{g}/\mathfrak{a})_0$ is a principal fibration with group
$\Sigma(\mathfrak{a})$, in particular, a Kan fibration (\cite{M}, Lemma 18.2).

\begin{lemma}\label{lemma: sigma kan}
Suppose that $\mathfrak{g}$ is a nilpotent $L_\infty$-algebra. Then,
$\Sigma(\mathfrak{g})$ is a Kan simplicial set.
\end{lemma}
\begin{proof}
If that $\mathfrak{g}$ is abelian then $\Sigma(\mathfrak{g})$ is a simplicial group and therefore a Kan simplicial set.

Let $F^\bullet\mathfrak{g}$ denote the lower central series. Assume that $Gr^i_F\mathfrak{g} \neq 0$ if and only if $0\leq i \leq n$; that is, $\mathfrak{g}$ is \emph{nilpotent of length $n$}. By induction assume that $\Sigma(\mathfrak{h})$ is a Kan simplicial set for any nilpotent $L_\infty$-algebra $\mathfrak{h}$ of length at most $n-1$.

Since $\mathfrak{g}$ is nilpotent of length $n$, it follows that $F^n\mathfrak{g} = Gr^n\mathfrak{g}$ is central in $\mathfrak{g}$ and $\mathfrak{g}/F^n\mathfrak{g}$ is nilpotent of length $n-1$. Therefore, $\Sigma(\mathfrak{g}/F^n\mathfrak{g}$ is a Kan simplicial set and so is $\Sigma(\mathfrak{g}/F^n\mathfrak{g})_0$. Since $\Sigma(\mathfrak{g}) \to \Sigma(\mathfrak{g}/F^n\mathfrak{g})_0$ is a Kan fibration it follows that $\Sigma(\mathfrak{g})$ is a Kan simplicial set as well.
\end{proof}

\begin{lemma}\label{lemma: vanishing of higher homotopy}
Suppose that $\mathfrak{g}$ is a nilpotent $L_\infty$-algebra such that $\mathfrak{g}^q =
0$ for $q\leq -k$, $k$ a positive integer. Then, for any connected component $X$ of $\Sigma(\mathfrak{g})$,  $\pi_i(X) = 0$ for $i>k$.
\end{lemma}
\begin{proof}
Suppose that $\mathfrak{g}$ is abelian. Then, $\pi_i\Sigma(\mathfrak{g}) \cong
H^{1-i}(\mathfrak{g})$. For $\mathfrak{g}$ not necessarily abelian the statement follows by induction on the nilpotency length, the abelian case establishing the base of the induction.

Let $F^\bullet\mathfrak{g}$ denote the lower central series. Assume that $Gr^i_F\mathfrak{g} \neq 0$ if and only if $0\leq i \leq n$; that is, $\mathfrak{g}$ is \emph{nilpotent of length $n$}. By induction assume that the conclusion holds for all nilpotent $L_\infty$-algebras of length at most $n-1$.

Since $\mathfrak{g}$ is nilpotent of length $n$, it follows that $F^n\mathfrak{g} = Gr^n\mathfrak{g}$ is central in $\mathfrak{g}$ and $\mathfrak{g}/F^n\mathfrak{g}$ is nilpotent of length $n-1$. Let $X\subseteq\Sigma(\mathfrak{g})$ be a connected component of $\Sigma(\mathfrak{g})$ and let $Y\subseteq\Sigma(\mathfrak{g}/F^n\mathfrak{g})$ be the image of $X$ under the map induced by the quotient map $\mathfrak{g} \to \mathfrak{g}/F^n\mathfrak{g}$. Then, $X \to Y$ is a principal fibration with group the connected component of the identity in $\Sigma(F^n\mathfrak{g})$. The desired vanishing of higher homotopy groups of $X$ follows from the induction hypotheses using the long exact sequence of homotopy groups.
\end{proof}

\subsubsection{Homotopy invariance}

\begin{lemma}\label{lemma: quism ab alg equiv on sigma}
Suppose that $f \colon \mathfrak{a} \to \mathfrak{b}$ is a quasi-isomorphism of abelian
algebras. Then, the induced map $\Sigma(f) \colon \Sigma(\mathfrak{a}) \to
\Sigma(\mathfrak{b})$ is an equivalence.
\end{lemma}
\begin{proof}
Note that $\Sigma(f)$ is a morphism of simplicial abelian groups.
It is sufficient to show that the maps $\pi_n\Sigma(f)\colon \pi_n\Sigma(\mathfrak{a}) \to \pi_n\Sigma(\mathfrak{b})$ are isomorphisms for $n \geqslant 0$. To this end note that $\pi_n\Sigma(f)$ factors as the composition of isomorphisms
\[
\pi_n\Sigma(\mathfrak{a}) \cong H^{1-n}(\mathfrak{a}) \xrightarrow{H^{1-n}(\Sigma(f))} H^{1-n}(\mathfrak{b}) \cong \pi_n\Sigma(\mathfrak{b}) .
\]
\end{proof}

\begin{prop}[\cite{G1}, Proposition 4.9]\label{prop:homotopy invariance}
Suppose that $f \colon \mathfrak{g} \to \mathfrak{h}$ is a quasi-isomorphism of
$L_\infty$-algebras and $R$ is an Artin algebra with maximal ideal $\mathfrak{m}_R$. Then,
the map $\Sigma(f\otimes\id) \colon \Sigma(\mathfrak{g}\otimes\mathfrak{m}_R) \to
\Sigma(\mathfrak{h}\otimes\mathfrak{m}_R)$ is an equivalence.
\end{prop}
\begin{proof}
We use induction on the nilpotency length of $\mathfrak{m}_R$, which is to say the largest integer $l$ such that $\mathfrak{m}_R^l \neq 0$.

If $\mathfrak{m}_R^2 = 0$, then $f\otimes\id \colon \mathfrak{g}\otimes\mathfrak{m}_R \to \mathfrak{h}\otimes\mathfrak{m}_R$ is a quasi-isomorphism of abelian algebras and the claim follows from Lemma \ref{lemma: quism ab alg equiv on sigma}.

Suppose that $\mathfrak{m}_R^{l+1} = 0$. By the induction hypothesis
\begin{itemize}
\item the map
$\Sigma(\mathfrak{g}\otimes\mathfrak{m}_R/\mathfrak{m}_R^l) \to
\Sigma(\mathfrak{h}\otimes\mathfrak{m}_R/\mathfrak{m}_R^l)$ is an equivalence and \item the map
$\pi_0\Sigma(\mathfrak{g}\otimes\mathfrak{m}_R/\mathfrak{m}_R^l) \to
\pi_0\Sigma(\mathfrak{h}\otimes\mathfrak{m}_R/\mathfrak{m}_R^l)$ is a bijection.
\end{itemize}
The map $f\otimes\id_{\mathfrak{m}_R^l}$ is a quasi-isomorphism of abelian
$L_\infty$-algebras, therefore the map $H^2(\mathfrak{g}\otimes\mathfrak{m}_R^l) \to
H^2(\mathfrak{h}\otimes\mathfrak{m}_R^l)$ is an isomorphism. The commutativity of
\[
\begin{CD}
\pi_0\Sigma(\mathfrak{g}\otimes\mathfrak{m}_R/\mathfrak{m}_R^l) @>>> \pi_0\Sigma(\mathfrak{h}\otimes\mathfrak{m}_R/\mathfrak{m}_R^l) \\
@VVV @VVV \\
H^2(\mathfrak{g}\otimes\mathfrak{m}_R^l) @>>> H^2(\mathfrak{h}\otimes\mathfrak{m}_R^l)
\end{CD}
\]
implies that the map
\[
\pi_0\Sigma(\mathfrak{g}\otimes\mathfrak{m}_R/\mathfrak{m}_R^l)_0 \to \pi_0\Sigma(\mathfrak{h}\otimes\mathfrak{m}_R/\mathfrak{m}_R^l)_0
\]
is a bijection. Therefore, the map
\[
\Sigma(\mathfrak{g}\otimes\mathfrak{m}_R/\mathfrak{m}_R^l)_0 \to \Sigma(\mathfrak{h}\otimes\mathfrak{m}_R/\mathfrak{m}_R^l)_0
\]
is an equivalence. The map $\Sigma(f)$ restricts to a map of principal fibrations
\[
\begin{CD}
\Sigma(\mathfrak{g}\otimes\mathfrak{m}_R) @>>> \Sigma(\mathfrak{h}\otimes\mathfrak{m}_R) \\
@VVV @VVV \\
\Sigma(\mathfrak{g}\otimes\mathfrak{m}_R/\mathfrak{m}_R^l)_0  @>>> \Sigma(\mathfrak{h}\otimes\mathfrak{m}_R/\mathfrak{m}_R^l)_0
\end{CD}
\]
relative to the map of simplicial groups $\Sigma(\mathfrak{g}\otimes\mathfrak{m}_R^l)
\to \Sigma(\mathfrak{h}\otimes\mathfrak{m}_R^l)$. The latter is an equivalence by Lemma
\ref{lemma: quism ab alg equiv on sigma}. Therefore, so is the map
$\Sigma(\mathfrak{g}\otimes\mathfrak{m}_R) \to
\Sigma(\mathfrak{h}\otimes\mathfrak{m}_R)$.
\end{proof}

\subsection{Deligne groupoids}\label{ss: Deligne}

\subsubsection{Gauge transformations}
Suppose that $\mathfrak{h}$ is a nilpotent DGLA. Then, $\mathfrak{h}^0$ is a nilpotent Lie algebra. The unipotent group $\exp \mathfrak{h}^0$ acts on the space $\mathfrak{h}^1$ by affine transformations. The action of $\exp X$, $X\in\mathfrak{h}^0$, on $\gamma\in\mathfrak{h}^1$ is given by the formula
\begin{equation}\label{gauge transformation}
(\exp X) \cdot \gamma= \gamma- \sum_{i=0}^{\infty} \frac{(\ad
X)^i}{(i+1)!}(\delta X + [\gamma, X]) .
\end{equation}
The effect of the above action on the curvature $\mathcal{F}(\gamma) = \delta\gamma + \dfrac12 [\gamma,\gamma]$ is given by
\begin{equation}\label{gauge transformation - curva}
\mathcal{F}((\exp X) \cdot \gamma) = \exp(\ad X)(\mathcal{F}(\gamma)) .
\end{equation}

\subsubsection{The functor $\MC^1$}\label{ss:MC1}
Suppose that $\mathfrak{h}$ is a nilpotent DGLA.
It follows from \eqref{gauge transformation - curva} that gauge transformations \eqref{gauge transformation} preserve the subset of Maurer-Cartan elements $\MC(\mathfrak{h})\subset\mathfrak{h}^1$.

We denote by $\MC^1(\mathfrak{h})$ the Deligne groupoid
(denoted $\mathcal{C}(\mathfrak{h})$ in \cite{H1}) defined as the groupoid associated with the action of the group $\exp \mathfrak{h}^0$ by gauge
transformations on the set $\MC(\mathfrak{h})$.

Thus, $\MC^1(\mathfrak{h})$ is the category with the set of objects $\MC(\mathfrak{h})$. For $\gamma_1, \gamma_2 \in \MC(\mathfrak{h})$, $\Hom_{\MC^1(\mathfrak{h})}(\gamma_1, \gamma_2)$ is the set of gauge transformations between $\gamma_1$, $\gamma_2$. The composition
\begin{equation*}
\Hom_{\MC^1(\mathfrak{h})}(\gamma_2,
\gamma_3)\times\Hom_{\MC^1(\mathfrak{h})}(\gamma_1,
\gamma_2)\to\Hom_{\MC^1(\mathfrak{h})}(\gamma_1, \gamma_3)
\end{equation*}
is given by  the product in the group $\exp(\mathfrak{h}^0)$.

\subsubsection{The functor $\MC^2$}\label{ss:MC2}
For $\mathfrak{h}$ as above satisfying the additional vanishing condition $\mathfrak{h}^i
= 0$ for $i < -1$ we denote by $\MC^2(\mathfrak{h})$ the Deligne $2$-groupoid as
defined by P.~Deligne \cite{Del} and independently by E.~Getzler, \cite{G1}. Below we review the construction of Deligne $2$-groupoid of a nilpotent DGLA following \cite{G1, G2} and references therein.

The objects and the $1$-morphisms of $\MC^2(\mathfrak{h})$ are those of $\MC^1(\mathfrak{h})$. That is, for $\gamma_1, \gamma_2 \in \MC(\mathfrak{h})$ the set $\Hom_{\MC^1(\mathfrak{h})}(\gamma_1, \gamma_2)$ is the set of objects of the groupoid   $\Hom_{\MC^2(\mathfrak{h})}(\gamma_1, \gamma_2)$. The morphisms in $\Hom_{\MC^2(\mathfrak{h})}(\gamma_1, \gamma_2)$ (i.e. the $2$-morphisms of $\MC^2(\mathfrak{h})$) are defined as follows.

For $\gamma \in \MC(\mathfrak{h})$ let $[\cdot, \cdot]_{\gamma}$ denote the Lie bracket on $\mathfrak{h}^{-1}$ defined by
\begin{equation}\label{eq:mu-bracket}
[a,\,b]_{\gamma}=[a,\, \delta b+[\gamma, \,b]].
\end{equation}
Equipped with this bracket, $\mathfrak{h}^{-1}$ becomes a nilpotent Lie
algebra. We denote by $\exp_{\gamma} \mathfrak{h}^{-1}$ the
corresponding unipotent group, and by
\[
\exp_{\gamma} \colon \mathfrak{h}^{-1} \to \exp_{\gamma}\mathfrak{h}^{-1}
\]
the corresponding exponential map. If $\gamma_1$, $\gamma_2$ are two Maurer-Cartan
elements, then the group $\exp_{\gamma_2} \mathfrak{h}^{-1}$ acts on
$\Hom_{\MC^1(\mathfrak{h})}(\gamma_1, \gamma_2)$. For $\exp_{\gamma_2} t
\in \exp_{\gamma_2} \mathfrak{h}^{-1}$ and $\Hom_{\MC^1(\mathfrak{h})}(\gamma_1, \gamma_2)$ the action is given by
\begin{equation*}
(\exp_{\gamma_2} t) \cdot (\exp X) = 
\exp(\delta t+[\gamma_2,t])\, \exp X \in \exp
\mathfrak{h}^0 .
\end{equation*}
By definition, $\Hom_{\MC^2(\mathfrak{h})}(\gamma_1, \gamma_2)$ is the groupoid associated with the above action.


The horizontal composition in $\MC^2(\mathfrak{h})$, i.e. the map of groupoids
\begin{multline*}
\otimes: \Hom_{\MC^2(\mathfrak{h})}(\exp X_{23}, \exp Y_{23}) \times
\Hom_{\MC^2(\mathfrak{h})}(\exp X_{12}, \exp Y_{12}) \to\\
\Hom_{\MC^2(\mathfrak{h})}(\exp X_{23}\exp X_{12}, \exp X_{23}\exp
Y_{12}) ,
\end{multline*}
where $\gamma_i \in \MC(\mathfrak{h})$, $\exp X_{ij}, \exp Y_{ij}$, $1\leq i,j\leq 3$ is defined by
\[
\exp_{\gamma_{3}} t_{23} \otimes \exp_{\gamma_{2}} t_{12}
=\exp_{\gamma_{3}} t_{23}\exp_{\gamma_3}( \exp(\ad X_{23})(t_{12}) ) ,
\]
where $\exp_{\gamma_j} t_{ij} \in \Hom_{\MC^2(\mathfrak{h})}(\exp X_{ij}, \exp Y_{ij})$.

\begin{remark}\label{remark: MC1 to MC2}
There is a canonical map of $2$-groupoids $\MC^1(\mathfrak{h}) \to
\MC^2(\mathfrak{h})$ which induces a bijection $\pi_0(\MC^1(\mathfrak{h})) \to
\pi_0(\MC^2(\mathfrak{h}))$ on sets of isomorphism classes of objects.
\end{remark}

\subsection{Properties of $\mathfrak{N}\MC^2$}

\subsubsection{Abelian algebras}

\begin{lemma}\label{nerve mc2 is eilenberg-maclane for abelian}
Suppose that $\mathfrak{a}$ is an abelian DGLA satisfying $\mathfrak{a}^i = 0$ for $i<
-1$. Then, the simplicial sets $\mathfrak{N}\MC^2(\mathfrak{a})$ and
$K(\mathfrak{a}[1])$ are isomorphic naturally in $\mathfrak{a}$.
\end{lemma}
\begin{proof}
The claim is an immediate consequence of the definitions and the explicit description of the nerve of $\MC^2(\mathfrak{a})$ given in Lemma \ref{lemma:the two gothic Ns}.
\end{proof}

Combining Lemma \ref{nerve mc2 is eilenberg-maclane for abelian} with the integration map
\eqref{abelian integration} we obtain the map of simplicial abelian groups
\begin{equation}\label{sigma int nerve mc2 abelian}
\int\colon \Sigma(\mathfrak{a}) \to \mathfrak{N}\MC^2(\mathfrak{a})
\end{equation}
which is a homotopy equivalence.

\subsubsection{Central extensions}
Suppose that $\mathfrak{g}$ is a nilpotent DGLA satisfying $\mathfrak{g}^i = 0$ for $i< -1$ and $\mathfrak{a}$ is a central subalgebra in $\mathfrak{g}$. Note that $\MC^2$ commutes with products, $\mathfrak{N}$ commutes with products and the addition map $+\colon \mathfrak{a}\times\mathfrak{g} \to \mathfrak{g}$ is a morphism of DGLA.
Thus, we obtain an action of the simplicial abelian group
$\mathfrak{N}\MC^2(\mathfrak{a})$ on the simplicial set
$\mathfrak{N}\MC^2(\mathfrak{g})$
\[
\mathfrak{N}\MC^2(+)\colon \mathfrak{N}\MC^2(\mathfrak{a})\times\mathfrak{N}\MC^2(\mathfrak{g}) \to \mathfrak{N}\MC^2(\mathfrak{g}) .
\]
Note that the group structure on $\mathfrak{N}\MC^2(\mathfrak{a})$ is obtained from the case $\mathfrak{a} = \mathfrak{g}$. Clearly, the action is free and the map
$\mathfrak{N}\MC^2(\mathfrak{g}) \to
\mathfrak{N}\MC^2(\mathfrak{g}/\mathfrak{a})$ factors through
$\mathfrak{N}\MC^2(\mathfrak{g})/\mathfrak{N}\MC^2(\mathfrak{a})$.

\subsubsection{The obstruction map}
\begin{lemma}\label{lemma: iso invariance of obst and gemuse}
The obstruction map \eqref{eqn:o2} factors as
\[
\MC(\mathfrak{g}/\mathfrak{a}) \to \pi_0\MC^2(\mathfrak{g}/\mathfrak{a}) \to H^2(\mathfrak{a})
\]
\end{lemma}
\begin{proof}
Suppose $\mu+\mathfrak{a}^1 \in \MC(\mathfrak{g}/\mathfrak{a})$. It follows from the formula \eqref{gauge transformation} that
\[
\exp(X+\mathfrak{a}^0)\cdot(\mu+\mathfrak{a}^1) = (\exp X)\cdot\mu + \mathfrak{a}^1 .
\]
The formula \eqref{gauge transformation - curva} implies  that
\[
\mathcal{F}(\exp(X+\mathfrak{a}^0)\cdot(\mu+\mathfrak{a}^1)) = \mathcal{F}((\exp X)\cdot\mu) + \delta\mathfrak{a}^1 = \exp(\ad X)(\mathcal{F}(\mu) + \delta\mathfrak{a}^1) .
\]
Since $\mathcal{F}(\mu) + \delta\mathfrak{a}^1 \subset \mathfrak{a}^2$, it follows that $\exp(\ad X)(\mathcal{F}(\mu) + \delta\mathfrak{a}^1) = \mathcal{F}(\mu) + \delta\mathfrak{a}^1$ or, equivalently, $o_2(\exp(X+\mathfrak{a}^0)\cdot(\mu+\mathfrak{a}^1)) = o_2(\mu+\mathfrak{a}^1)$.
\end{proof}

Recall (Lemma \ref{lemma:the two gothic Ns}) that an $n$-simplex of $\mathfrak{N}\MC^2(\mathfrak{g}/\mathfrak{a})$, i.e. an element of $\mathfrak{N}_n\MC^2(\mathfrak{g}/\mathfrak{a})$ includes, among other things, a collection of $n+1$ gauge-equivalent Maurer-Cartan elements of $\mathfrak{g}/\mathfrak{a}$. By Lemma \ref{lemma: iso invariance of obst and gemuse} all of these Maurer-Cartan elements give rise to the same element of $H^2(\mathfrak{a})$ under the map \eqref{eqn:o2}. Therefore, the assignment of this common value to an element of $\mathfrak{N}_n\MC^2(\mathfrak{g}/\mathfrak{a})$ give rise to a well-defined map
\begin{equation}\label{eq: Obst und Gemuse 2}
o_{2,n}\colon \mathfrak{N}_n\MC^2(\mathfrak{g}/\mathfrak{a}) \to H^2(\mathfrak{a})
\end{equation}
for each $n = 0, 1, 2, \ldots$ such that the sequence of pointed sets
\[
0 \to \mathfrak{N}_n\MC^2(\mathfrak{g})/\mathfrak{N}_n\MC^2(\mathfrak{a}) \to \mathfrak{N}_n\MC^2(\mathfrak{g}/\mathfrak{a}) \xrightarrow{o_{2,n}} H^2(\mathfrak{a})
\]
is exact. The maps \eqref{eq: Obst und Gemuse 2} assemble into a map of simplicial sets
\[
o_2 \colon \mathfrak{N}\MC^2(\mathfrak{g}/\mathfrak{a}) \xrightarrow{o_2} H^2(\mathfrak{a}) ,
\]
where $H^2(\mathfrak{a})$ is constant. Let
$\mathfrak{N}\MC^2(\mathfrak{g}/\mathfrak{a})_0 = o_2^{-1}(0)$. The simplicial subset
$\mathfrak{N}\MC^2(\mathfrak{g}/\mathfrak{a})_0$ is a union of connected
components of $\mathfrak{N}\MC^2(\mathfrak{g}/\mathfrak{a})$ equal to the range of
the map $\mathfrak{N}\MC^2(\mathfrak{g})/\mathfrak{N}\MC^2(\mathfrak{a}) \to
\mathfrak{N}\MC^2(\mathfrak{g}/\mathfrak{a})$.

It follows that $\mathfrak{N}\MC^2(\mathfrak{g}) \to
\mathfrak{N}\MC^2(\mathfrak{g}/\mathfrak{a})_0$ is a principal fibration with the group
$\mathfrak{N}\MC^2(\mathfrak{a})$.

\section{$\mathfrak{N}\MC^2$ vs. $\Sigma$}\label{section: mc vs sigma}
In this section we show that for a DGLA $\mathfrak{h}$ satisfying $\mathfrak{h}^i = 0$ for $i< -1$ the simplicial sets $\mathfrak{N}\MC^2(\mathfrak{h})$ and $\Sigma(\mathfrak{h})$ are isomorphic in the homotopy category of simplicial sets.

\subsection{The main theorem}
Let $\mathbf{\Sigma}^2_n(\mathfrak{h}) = \widetilde{\MC^2(\Omega_n\otimes\mathfrak{h})}$, where the latter is the simplicial groupoid associated with the strict 2-groupoid $\MC^2(\Omega_n\otimes\mathfrak{h})$ (see \ref{sss:FStSG}).
Let $\mathbf{\Sigma}^2(\mathfrak{h}) \colon [n] \mapsto
\mathbf{\Sigma}^2_n(\mathfrak{h})$ denote the corresponding simplicial object in
simplicial groupoids. Note that $\Sigma(\mathfrak{h})$ is the simplicial set of objects of
$\mathbf{\Sigma}^2(\mathfrak{h})$, hence there is a canonical map
\begin{equation}\label{Sigma to simpl nerve SIGMA-2}
\Sigma(\mathfrak{h}) \to \mathfrak{N}\mathbf{\Sigma}^2(\mathfrak{h}) .
\end{equation}

The map $\mathbb{Q} \to \Omega_\bullet$ of simplicial DGA induces the map of simplicial objects in simplicial groupoids
\begin{equation}\label{MC2 to SIGMA-2}
\MC^2(\mathfrak{h}) \to \mathbf{\Sigma}^2(\mathfrak{h}) .
\end{equation}

Consider the diagram
\begin{equation}\label{Sigma to simpl nerve SIGMA-2 from MC2}
\begin{CD}
\Sigma(\mathfrak{h}) @>{\eqref{Sigma to simpl nerve SIGMA-2}}>> \mathfrak{N}\mathbf{\Sigma}^2(\mathfrak{h}) @<{\mathfrak{N}\eqref{MC2 to SIGMA-2}}<< \mathfrak{N}\MC^2(\mathfrak{h}) .
\end{CD}
\end{equation}

\begin{thm}\label{thm: nerve is sigma}
Suppose that $\mathfrak{h}$ is a nilpotent DGLA satisfying $\mathfrak{h}^i = 0$ for $i < -1$. Then, the morphisms \eqref{Sigma to simpl nerve SIGMA-2} and $\mathfrak{N}\eqref{MC2 to SIGMA-2}$ are equivalences so that the diagram \eqref{Sigma to simpl nerve SIGMA-2 from MC2} represents an
isomorphism  $\Sigma(\mathfrak{h})\cong\mathfrak{N}\MC^2(\mathfrak{h})$ in the
homotopy category of simplicial sets.
\end{thm}

The rest of Section \ref{section: mc vs sigma} is devoted to a proof of Theorem
\ref{thm: nerve is sigma} which borrows techniques from the proof of Proposition 3.2.1 of
\cite{H4}.

\subsection{The map \eqref{Sigma to simpl nerve SIGMA-2} is an equivalence}
Let $\mathbf{\Sigma}^1(\mathfrak{h})$ denote the simplicial object in groupoids defined
by $\mathbf{\Sigma}^1_n(\mathfrak{h}) = \MC^1(\Omega_n\otimes\mathfrak{h})$.
Note that $\Sigma(\mathfrak{h})$ is the simplicial set of objects of
$\mathbf{\Sigma}^1(\mathfrak{h})$ and hence there is a canonical map
\begin{equation}\label{Sigma to nerve SIGMA-1}
\Sigma(\mathfrak{h}) \to \mathcal{N}\mathbf{\Sigma}^1(\mathfrak{h}) ;
\end{equation}
by Remark \ref{remark: MC1 to MC2} there is a canonical map of simplicial objects in
simplicial groupoids
\begin{equation}\label{Sigma-1 to SIGMA-2}
\mathbf{\Sigma}^1(\mathfrak{h}) \to \mathbf{\Sigma}^2(\mathfrak{h}) .
\end{equation}
The map \eqref{Sigma to simpl nerve SIGMA-2} is equal to the composition
\[
\Sigma(\mathfrak{h}) \xrightarrow{\eqref{Sigma to nerve SIGMA-1}} \mathcal{N}\mathbf{\Sigma}^1(\mathfrak{h}) \xrightarrow{\mathcal{N}\eqref{Sigma-1 to SIGMA-2} } \mathcal{N}\mathbf{\Sigma}^2(\mathfrak{h}) \to \mathfrak{N}\mathbf{\Sigma}^2(\mathfrak{h}) ,
\]
where the last map is the equivalence of Theorem \ref{thm: comparison of nerves}.

\begin{lemma}[\cite{H4}, Proposition 3.2.1]
The map \eqref{Sigma to nerve SIGMA-1} is an equivalence.
\end{lemma}
\begin{proof}
Let $G_n(\mathfrak{h}) := \exp((\Omega_n\otimes\mathfrak{h})^0)$. Then,
$G(\mathfrak{h}) \colon [n] \mapsto G_n(\mathfrak{h})$ is a simplicial group acting on
$\Sigma(\mathfrak{h})$, and $\mathbf{\Sigma}(\mathfrak{h})$ is the associated
groupoid. Therefore,
\[
N_q\mathbf{\Sigma}(\mathfrak{h}) = \Sigma(\mathfrak{h})\times G(\mathfrak{h})^{\times q}
\]
and the map
\[
\Sigma(\mathfrak{h}) \to N_q\mathbf{\Sigma}(\mathfrak{h})
\]
is an equivalence because $G(\mathfrak{h})$ is contractible.
\end{proof}

\begin{prop}
The map $\mathcal{N}\eqref{Sigma-1 to SIGMA-2}$ is an equivalence.
\end{prop}
\begin{proof}
Let $\mathbf{\Gamma}^1(\mathfrak{h})$ (respectively, $\mathbf{\Gamma}^2(\mathfrak{h})$) denote the full subcategory of
$\mathbf{\Sigma}^1(\mathfrak{h})$ (respectively, of $\mathbf{\Sigma}^2(\mathfrak{h})$) whose set of objects is $\MC(\mathfrak{h})$ (a constant simplicial set). There is
a commutative diagram
\[
\begin{CD}
\mathbf{\Gamma}^1(\mathfrak{h}) @>>> \mathbf{\Gamma}^2(\mathfrak{h}) \\
@VVV @VVV \\
\mathbf{\Sigma}^1(\mathfrak{h}) @>{\eqref{Sigma-1 to SIGMA-2}}>> \mathbf{\Sigma}^2(\mathfrak{h})
\end{CD}
\]
The vertical arrows induce equivalences on respective nerves since, for each $n$ the functors
$\mathbf{\Gamma}^1(\mathfrak{h})_n \to \mathbf{\Sigma}^1(\mathfrak{h})_n =
\MC^1(\Omega_n\otimes\mathfrak{h})$ and $\mathbf{\Gamma}^2(\mathfrak{h})_n
\to \mathbf{\Sigma}^2(\mathfrak{h})_n = \MC^2(\Omega_n\otimes\mathfrak{h})$ are
equivalences by \cite{H3}, Proposition 8.2.5.

The map $\mathbf{\Gamma}^1(\mathfrak{h}) \to \mathbf{\Gamma}^2(\mathfrak{h})$
induces a bijection between sets of isomorphism classes of objects. For
$\mu\in\MC(\mathfrak{h})$, $\Hom_{\mathbf{\Gamma}^2(\mathfrak{h})}(\mu,\mu)$
is naturally identified with the nerve of the groupoid associated to the action of the simplicial
group $H(\mathfrak{h},\mu) \colon [n] \mapsto
\exp((\Omega_n\otimes\mathfrak{h})_\mu)$ on the simplicial set
$\Hom_{\mathbf{\Gamma}^1(\mathfrak{h})}(\mu,\mu)$. Since the group
$H(\mathfrak{h},\mu)$ is contractible (it is isomorphic as a simplicial set to $[n] \mapsto
\Omega_n^0\otimes\mathfrak{h}^{-1}$) the induced map
$\Hom_{\mathbf{\Gamma}^1(\mathfrak{h})}(\mu,\mu) \to
\Hom_{\mathbf{\Gamma}^2(\mathfrak{h})}(\mu,\mu)$ is an equivalence.
\end{proof}

\subsection{The map $\mathfrak{N}(\eqref{MC2 to SIGMA-2}) \colon
\mathfrak{N}\MC^2(\mathfrak{h}) \to \mathfrak{N}\mathbf{\Sigma}^2(\mathfrak{h})$
is an equivalence}
It suffices to show that the map
\[
\mathfrak{N}\MC^2(\mathfrak{h}) \to \mathfrak{N}\MC^2(\Omega_n\otimes\mathfrak{h})
\]
is an equivalence for all $n$. This follows from Proposition \ref{prop: g to g-omega on mc2
equiv}.

\begin{prop}\label{prop: g to g-omega on mc2 equiv}
Suppose that $\mathfrak{h}$ is a nilpotent DGLA concentrated in degrees greater than or
equal to $-1$. The functor
\begin{equation}\label{g to g-omega on mc}
\MC^2(\mathfrak{h}) \to \MC^2(\Omega_n\otimes\mathfrak{h})
\end{equation}
is an equivalence.
\end{prop}
\begin{proof}
The induced map $\pi_0(\eqref{g to g-omega on mc})$ is a bijection by Remark
\ref{remark: MC1 to MC2} and (the proof of) \cite{H1}, Lemma 2.2.1. The result now
follows from Lemma \ref{lemma: loops with and without forms} below.
\end{proof}

\begin{lemma}\label{lemma: loops with and without forms}
Suppose $\mu\in\MC(\mathfrak{h})$. The functor
\begin{equation}\label{equivalence on Homs}
\Hom_{\MC^2(\mathfrak{h})}(\mu, \mu) \to \Hom_{\MC^2(\Omega_n\otimes\mathfrak{h})}(\mu, \mu)
\end{equation}
is an equivalence.
\end{lemma}
\begin{proof}
According to the description given in \ref{ss:MC2}, for any nilpotent DGLA $(\mathfrak{g},\delta)$ with $\mathfrak{g}^i = 0$ for $i < -1$
and $\mu\in\MC(\mathfrak{g})$ the groupoid $\Hom_{\MC^2(\mathfrak{g})}(\mu,
\mu)$ is isomorphic to the groupoid associated with the action of the group
$\exp_\mu\mathfrak{g}^{-1}$ on the set $\exp(\ker(\delta_\mu^{-1})) \subset
\exp(\mathfrak{g}^0)$ where $\delta_\mu = \delta + [\mu, .]$.

Note that, for any $X\in\ker(\delta_\mu^{-1})$, the automorphism group $\Aut(\exp(X))$
is isomorphic to (the additive group) $\ker(\delta_\mu^{-1})$.

The map
\begin{equation}\label{inclusion constants}
([n]\to[0])^\ast\otimes\id \colon (\mathfrak{h}, \delta) \to (\Omega_n\otimes\mathfrak{h}, d + \delta)
\end{equation}
is a quasi-isomorphism of DGLA with the quasi-inverse given by the evaluation map $\ev_0 := ([0]\to[n])^\ast\otimes\id
\colon \Omega_n\otimes\mathfrak{h} \to \mathfrak{h}$ (for any choice of a morphism $[0]\to[n]$) which is a morphism of DGLA as well. The same maps are mutually quasi-inverse quasi-isomorphisms of DGLA
\[
(\mathfrak{h}, \delta_\mu) \rightleftarrows (\Omega_n\otimes\mathfrak{h}, d + \delta_\mu) .
\]

Since \eqref{inclusion constants} is a quasi-isomorphism and both DGLA are concentrated in
degrees greater than or equal to $-1$, the induced map $\ker(\delta_\mu^{-1}) \to \ker((d
+ \delta_\mu)^{-1})$ an isomorphism, hence so are the maps of automorphism groups.

Since the map \eqref{equivalence on Homs} admits a left inverse (namely, $\ev_0$) it
remains to show that the induced map on sets of isomorphism classes is surjective. Note
that, since $\ev_0$ is a surjective quasi-isomorphism, the map $d+\delta_\mu \colon
\ker(\ev_0)^{-1} \to \ker(\ev_0)^0\bigcap\ker((d + \delta_\mu)^0)$ is an isomorphism.

Consider $X\in (\Omega_n\otimes\mathfrak{g})^0$. Then, $X = \ev_0(X) + Y$ with
$Y\in\ker(\ev_0)$, and $(d+\delta_\mu)X = 0$ if and only if $\delta_\mu\ev_0(X) = 0$ and
$(d+\delta_\mu)Y = 0$.

Suppose $X\in \ker((d + \delta_\mu)^{0})$. Then, $\exp(X) =
\exp(\ev_0(X))\cdot\exp(Z)$ where $Z\in\ker(\ev_0)^0\bigcap\ker((d + \delta_\mu)^0)$,
and, therefore, $Z = (d+\delta_\mu)U$ for a \emph{uniquely determined} $U$.
\end{proof}

\end{document}